\documentclass{article}

\usepackage[T1]{fontenc}
\usepackage{lmodern}
\usepackage[a4paper]{geometry}
\usepackage{mathtools}
\usepackage{amsthm}
\usepackage{amssymb}
\usepackage[english]{babel}
\usepackage[colorlinks=true,citecolor=blue]{hyperref}

\hypersetup{pdfnewwindow}

\numberwithin{equation}{section}

\newtheorem{thm}{Theorem}[section]

\theoremstyle{remark}
\newtheorem{remark}{Remark}[section]
\newcommand{\ud}{\mathrm{d}}
\newcommand{\half}{{\textstyle{1\over2}}}

\newcommand{\R}{\mathbb{R}}

\newcommand{\eqdef}{\stackrel{\text{\tiny def}}{=}}

\begin{document}

%%%%%%%%%%%%%%%%%%%%%%%%%%%%%%%%%%%%%%%%%%%%%%%%%%%%%%%%%%%%%%%%%%%%%%%%
%%%%%%%%%%%%%%%%%%%%%%%%%%%%%%%%%%%%%%%%%%%%%%%%%%%%%%%%%%%%%%%%%%%%%%%%

\title{  \scalebox{0.93}{\bf On the singular limit of the Camassa--Holm equation}
}
  
\author{Billel Guelmame\thanks{New York University Abu Dhabi, United Arab Emirates,
\texttt{billel.guelmame@nyu.edu}}
}

%%%%%%%%%%%%%%%%%%%%%%%%%%%%%%%%%%%%%%%%%%%%%%%%%%%%%%%%%%%%%%%%%%%%%%%%
%%%%%%%%%%%%%%%%%%%%%%%%%%%%%%%%%%%%%%%%%%%%%%%%%%%%%%%%%%%%%%%%%%%%%%%%

\maketitle
%%%%%%%%%%%%%%%%%%%%%%%%%%%%%%%%%%%%%%%%%%%%%%%%%%%%%%%%%%%%%%%%%%%%%%%%

\begin{abstract}
We study the singular limit of the Camassa--Holm equation as the length scale $\ell$ tends to zero. 
Although the formal limit is the Burgers equation with flux $3u^2/2$, we show that, for a class of smooth initial data, convergence to entropy solutions fails after the shock formation, even along subsequences in local space-time $L^1$.
\end{abstract} 

%%%%%%%%%%%%%%%%%%%%%%%%%%%%%%%%%%%%%%%%%%%%%%%%%%%%%%%%%%%%%%%%%%%%%%%%
\medskip

\medskip

\tableofcontents

%%%%%%%%%%%%%%%%%%%%%%%%%%%%%%%%%%%%%%%%%%%%%%%%%%%%%%%%%%%%%%%%%%%%%%%%
\section{Introduction}
%%%%%%%%%%%%%%%%%%%%%%%%%%%%%%%%%%%%%%%%%%%%%%%%%%%%%%%%%%%%%%%%%%%%%%%%

We study the Camassa--Holm equation
\begin{equation}\label{CH high der}
u^\ell_t + 3 u^\ell u^\ell_x = \ell^2 \left( u^\ell_{xxt} + 2 u^\ell_x u^\ell_{xx} + u^\ell u^\ell_{xxx} \right),
\end{equation}
posed on $\R$, with $u^\ell(0,x)=u_0(x)$ and $\ell>0$.
The Camassa--Holm (CH) equation was introduced in \cite{CamassaHolm1993} as shallow-water model.
Its nonlinear transport and nonlocal structure allow both global smooth solutions and finite-time wave breaking: the solution remains bounded, while its slope becomes unbounded from below  \cite{ConstantinEscher1998Global,ConstantinEscher1998WaveBreaking}.
Here, we are interested in the behavior of solutions to \eqref{CH high der} when the length scale $\ell$ tends to zero.

Formally setting $\ell=0$ in \eqref{CH high der} gives the Burgers equation
\begin{equation}\label{B}
u_t + \left(\tfrac32 u^2 \right)_x = 0.
\end{equation}
For smooth initial data, the limit $\ell \to 0$ is justified on a common interval of classical existence. 
More precisely, Li, Yu, and Zhu \cite{LiYuZhu2023} proved that, for $u_0 \in H^s(\R)$ with $s>3/2$, there exists a time $T>0$, independent of $\ell \in (0,1)$, such that
\begin{equation*}
u^\ell \longrightarrow u \qquad \text{in }L^\infty([0,T]; H^s(\R)).
\end{equation*}
In a subsequent work \cite{LiYuZhu2024Nonuniform}, the same authors showed that this convergence is not uniform with respect to the initial datum in Sobolev spaces. These two results concern the smooth flows on a common short time interval. 

Related singular limits have been studied for viscous Camassa--Holm-type equations. 
Coclite and Karlsen \cite{CocliteKarlsen2006}, and later Coclite and di Ruvo \cite{CocliteDiRuvo2015}, proved convergence to the entropy solution of the Burgers equation requiring the Camassa--Holm regularization parameter $\ell$ to tend to zero sufficiently rapidly relative to the viscosity parameter.
These results rely essentially on the presence of viscosity and therefore do not apply to the singular limit of \eqref{CH high der}, in which no viscosity is present.

The question becomes different beyond the classical blow-up time of the Burgers equation. 
To state the result, define the Helmholtz operator and its Green function by
\begin{equation*}
\mathcal{H}_\ell = 1 - \ell^2 \partial_x^2, \qquad G_\ell(x) = \frac1{2\ell} \mathrm{e}^{-|x|/\ell},
\end{equation*}
so that $\mathcal{H}_\ell^{-1} v = G_\ell \ast v$.
With the momentum $m^\ell \eqdef \mathcal{H}_\ell u^\ell$, equation \eqref{CH high der} takes the form
\begin{equation}\label{momentum eq CH}
m_t^\ell + u^\ell m_x^\ell + 2 u_x^\ell m^\ell = 0.
\end{equation}
Applying $\mathcal{H}_\ell^{-1}$ to \eqref{CH high der} also gives
\begin{equation}\label{CH}
u_t^\ell + u^\ell u_x^\ell + \ell^2 P_x^\ell + Q_x^\ell  = 0,
\end{equation}
where
\begin{equation*}
P^\ell = G_\ell \ast \tfrac12 (u_x^\ell)^2 , \qquad Q^\ell = G_\ell \ast (u^\ell)^2.
\end{equation*}
Smooth solutions of \eqref{CH high der} conserve the $H^1$-like energy
\begin{equation*}
	\int_\R ((u^\ell)^2+\ell^2 (u_x^\ell)^2)\, \ud x.
\end{equation*}
The momentum equation preserves the sign of $m^\ell$, while the nonlocal formulation \eqref{CH} allows to study the existence of solutions and also to study the singular limit $\ell \to 0$.
We now state the main result of the paper.

%%%%%%%%%%%%%%%%%%%%%%%%%%%%%%%%%%%%%%%%%%%%%%%%%%%%%%%%%%%%%%%%%%%%%%%%
\begin{thm}\label{thm:main}
	Let $u_0 \in H^3(\R) \cap L^1(\R)$ be a non-trivial initial datum. Assume that there exists $\ell_0 \in (0,1)$ such that
\begin{equation*}
	m_0^\ell \eqdef u_0 - \ell^2 u_0'' \geqslant 0 \quad \text{(or $m_0^\ell \leqslant 0$)} \quad \text{for any } \ell \in (0,\ell_0).
\end{equation*}
Let $u^\ell$ be the corresponding global solution of \eqref{CH}, and let $u$ be the unique entropy solution of \eqref{B} with initial datum $u_0$. 
Set
\begin{equation*}
T_{\mathrm B}^* \eqdef -\frac{1}{3\inf_{x\in\R}u_0'(x)}\in(0,\infty).
\end{equation*}
Then, for every $t_2>t_1 \geqslant T_{\mathrm B}^*$, there is no sequence $\ell_n\downarrow0$ such that
\begin{equation}\label{convergence in L1}
u^{\ell_n} \to u \qquad \text{strongly in } L^1_{\mathrm loc}\bigl([t_1,t_2]\times \R \bigr).
\end{equation}
In particular, the Camassa--Holm solutions do not, in general, converge to the entropy solution of the Burgers equation after the shock formation.
\end{thm}
%%%%%%%%%%%%%%%%%%%%%%%%%%%%%%%%%%%%%%%%%%%%%%%%%%%%%%%%%%%%%%%%%%%%%%%%

Under the sign condition in Theorem \ref{thm:main}, the CH equation has a unique global solution \cite{ConstantinEscher1998Global,RodriguezBlanco2001}, see also Theorem \ref{thm:CH-global-momentum} below.
On another side, the time $T_{\mathrm B}^*$ is the blow-up time of smooth solutions to the Burgers equation and it is finite because a nontrivial $H^3(\R)$ function, which tends to zero at both ends of the real line, must have a negative slope somewhere. 
The conclusion concerns convergence to the unique entropy solution. 
This theorem does not assert that the Camassa--Holm family has no strongly convergent subsequence with a different limit.

%%%%%%%%%%%%%%%%%%%%%%%%%%%%%%%%%%%%%%%%%%%%%%%%%%%%%%%%%%%%%%%%%%%%%%%%
\begin{remark}
An explicit initial datum satisfying the assumptions of Theorem \ref{thm:main} is
\begin{equation*}
	u_0(x)= \mathrm{sech} (x)=\frac{2}{\mathrm{e}^x + \mathrm{e}^{-x}}, 
\end{equation*}
and therefore 
\begin{equation*}
	u'_0(x)= \frac{-2(\mathrm{e}^x - \mathrm{e}^{-x})}{(\mathrm{e}^x + \mathrm{e}^{-x})^2}, 
	\qquad 
	u''_0(x)= \frac{-2}{\mathrm{e}^x + \mathrm{e}^{-x}} +\frac{4(\mathrm{e}^x - \mathrm{e}^{-x})^2}{(\mathrm{e}^x + \mathrm{e}^{-x})^3}.
\end{equation*}
Using that $(\mathrm{e}^x - \mathrm{e}^{-x})^2=(\mathrm{e}^x + \mathrm{e}^{-x})^2-4$, we obtain
\begin{equation*}
	m_0^\ell = u_0(x) - \ell^2 u''_0(x) = (1-\ell^2) \mathrm{sech} (x) + 2 \ell^2 \mathrm{sech}^3 (x) \geqslant 0 \quad \text{for any } \ell \in (0,1).
\end{equation*}
\end{remark}
%%%%%%%%%%%%%%%%%%%%%%%%%%%%%%%%%%%%%%%%%%%%%%%%%%%%%%%%%%%%%%%%%%%%%%%%

We briefly explain the proof. 
The sign condition gives the pointwise estimate $\ell |u_x^\ell| \leqslant |u^\ell|$. 
Together with conservation of mass and of the cubic Hamiltonian
\begin{equation*}
H_\ell(t) \eqdef \int_\R \bigl((u^\ell)^3 +\ell^2 u^\ell(u_x^\ell)^2 \bigr)\, \ud x,
\end{equation*}
this yields uniform $L^p(\R)$ bounds for $u^\ell$ and $\ell u_x^\ell$, for $p \in [1,3]$.
These bounds also control the local energy flux in $L^1(\R)$.
Testing the local energy equation with spatial cutoffs then shows that the energy outside $[-R,R]$ tends to zero as $R \to \infty$, uniformly in $\ell$ on bounded time intervals.

Suppose now that \eqref{convergence in L1} holds, and write
$I=[t_1,t_2]$. 
The uniform $L^3$ bound upgrades this convergence to strong local $L^2$ convergence. 
The energy tightness then gives strong convergence in $L^2(I \times \R)$. 
Since the Burgers entropy solution has lost quadratic energy after $T_{\mathrm B}^*$, whereas the CH solutions conserve their total energy, a strictly positive amount of energy must remain in $\ell^2(u_x^\ell)^2$. 
More precisely, after extraction of a subsequence, $w^\ell \eqdef \ell^2 (u_x^\ell)^2$ converge weakly to a nonnegative function $w$ with positive mass.
The energy tightness is used a second time here, it prevents this mass from escaping to spatial infinity. 

Finally, strong $L^2$ convergence gives $Q^\ell \rightharpoonup u^2$ and $\ell^2 P^\ell \rightharpoonup \half w$ in the sense of distributions.
Passing to the limit in \eqref{CH}, we obtain
\begin{equation*}
u_t + \left(\tfrac32 u^2\right)_x = -\half w_x \qquad \text{on } (t_1,t_2) \times \R.
\end{equation*}
Since $u$ solves \eqref{B}, we must have $w_x=0$.
Thus, $w$ must vanish because it has a finite mass for almost any time.
This contradicts the positive mass of $w$.

For comparison, consider the Hamiltonian regularized Burgers (rB) equation introduced in \cite{GuelmameJuncaClamondPego2024}
\begin{equation*}
v^\ell_t + v^\ell v^\ell_x = \ell^2 \left( v^\ell_{xxt} + 2 v^\ell_x v^\ell_{xx} + v^\ell v^\ell_{xxx} \right).
\end{equation*}
The right-hand side is the same as in \eqref{CH high der}, but the coefficient of the transport term is different. 
Both equations have a local smooth theory with a time of existence uniform in $\ell$ for fixed sufficiently regular data, and their smooth flows converge locally in time to the corresponding classical Burgers flows \cite{LiYuZhu2023,GuelmameJuncaClamondPego2024}.
Their smooth solutions conserve the same form of $H^1$-like energy.
However, the rB equation is Galilean invariant, whereas the CH equation is not. 
Moreover, only one Hamiltonian structure is known for the rB equation, in contrast with the bi-Hamiltonian structure of the CH equation.

A further difference is particularly relevant here. 
Like the Burgers equation, rB develops wave breaking in finite time whenever a smooth initial datum has a negative slope somewhere, with lifespan bounds independent of $\ell$ \cite{GuelmameJuncaClamondPego2024}.
For the CH equation, a negative slope alone does not force wave breaking, initial momentum of one sign gives global smooth solutions (see Theorems \ref{thm:Blow-up} and \ref{thm:CH-global-momentum} below).
This difference is precisely what allows the conservative dynamics used in Theorem \ref{thm:main} to persist beyond Burgers shock formation.

Both equations admit global dissipative weak solutions for $H^1$ initial data
\cite{BressanConstantin2007Dissipative,GuelmameJuncaClamondPego2024,Guelmame2024Hamiltonian}.
For the dissipative CH flow, the available one-sided slope bounds depend on $\ell$ and may blow up as $\ell \to 0$. 
In contrast, dissipative rB solutions satisfy the uniform Oleinik estimate
\begin{equation*}
v_x^\ell(t,x) \leqslant \tfrac2t \qquad \text{for almost every }(t,x) \in (0,\infty) \times \R.
\end{equation*}
This estimate is an important ingredient in the analysis of the rB limit as $\ell \to 0$.

Suitable dissipative solutions recover the entropy solution of the Burgers equation
\begin{equation}\label{cB}
v_t + \left(\half v^2\right)_x = 0	
\end{equation}
even after shock formation \cite{Guelmame2024Hamiltonian}.
Indeed, for $v_0 \in H^1(\R)$ with $v_0' \in L^1(\R)$ and $v_0' \leqslant M <\infty$, the dissipative solutions obtained by the vanishing viscosity method satisfy \cite{GuelmameHouamed2024}
\begin{equation*}
\|v^\ell-v\|_{L^\infty([0,T];L^p(\R))}
\leqslant C\ell^{1/(2p)},\qquad 0<\ell\leqslant1,
\end{equation*}
for every $T>0$ and $1 \leqslant p < \infty$, where $C$ is independent of $\ell$, and $v$ is the unique entropy solution of the Burgers equation \eqref{cB}.
Thus the common energy conservation law for smooth solutions does not determine the singular limit. 
The possibility of wave breaking and the dissipation in the continuation also matter. 
The rB equation admits dissipative solutions converging to entropy solutions of the Burgers equation, while the global smooth CH flows considered here fail to converge to the Burgers entropy solution even in $L^1_{\mathrm{loc}}$.

Section \ref{sec:preliminaries} recalls the needed properties of the Burgers and Camassa--Holm equations, while Section \ref{sec:proof} is devoted to proving Theorem \ref{thm:main}.

%%%%%%%%%%%%%%%%%%%%%%%%%%%%%%%%%%%%%%%%%%%%%%%%%%%%%%%%%%%%%%%%%%%%%%%%
\section{Preliminaries}\label{sec:preliminaries}
%%%%%%%%%%%%%%%%%%%%%%%%%%%%%%%%%%%%%%%%%%%%%%%%%%%%%%%%%%%%%%%%%%%%%%%%

All equations below are posed on $\R$. 
The parameter $\ell>0$ is fixed unless a limit as $\ell\to0$ is explicitly considered. 
We recall the results needed for the proof below.

%%%%%%%%%%%%%%%%%%%%%%%%%%%%%%%%%%%%%%%%%%%%%%%%%%%%%%%%%%%%%%%%%%%%%%%%
\subsection{The Burgers equation}
%%%%%%%%%%%%%%%%%%%%%%%%%%%%%%%%%%%%%%%%%%%%%%%%%%%%%%%%%%%%%%%%%%%%%%%%

We first recall that for any smooth initial datum, the Burgers equation admits a unique local smooth solution. 
Moreover, if the initial datum is decreasing somewhere, then, the corresponding solution blows up in finite time. 

%%%%%%%%%%%%%%%%%%%%%%%%%%%%%%%%%%%%%%%%%%%%%%%%%%%%%%%%%%%%%%%%%%%%%%%%
\begin{thm}[Classical Burgers flow and breakdown]\label{thm:Burgers-LWP}
Let $s>\frac32$ and $u_0\in H^s(\R)$.  Equation \eqref{B} has a unique
maximal solution
\begin{equation*}
 u \in C\bigl([0,T_{\mathrm B}^*); H^s(\R)\bigr)  \cap C^1 \bigl([0,T_{\mathrm B}^*); H^{s-1}(\R)\bigr).
\end{equation*}
Moreover, 
\begin{equation*}
 T_{\mathrm B}^*=
 \begin{cases}
  -\dfrac{1}{3 \inf_{x \in \R} u_0'(x)},     & \text{if } \inf_{x \in \R}u_0'(x) < 0, \\
  \infty, & \text{if } \inf_{x \in \R} u_0'(x)\geqslant 0.
 \end{cases}
\end{equation*}
In the first case, we have
\begin{equation*}
 \lim_{t\uparrow T_{\mathrm B}^*}\inf_{x\in\R}u_x(t,x)=-\infty.
\end{equation*}
\end{thm}
%%%%%%%%%%%%%%%%%%%%%%%%%%%%%%%%%%%%%%%%%%%%%%%%%%%%%%%%%%%%%%%%%%%%%%%%

After the blow-up time, we can consider weak solutions in the sense of distributions. 
Distributional solutions need not be unique. 
Imposing the entropy inequalities recovers uniqueness.

%%%%%%%%%%%%%%%%%%%%%%%%%%%%%%%%%%%%%%%%%%%%%%%%%%%%%%%%%%%%%%%%%%%%%%%%
\begin{thm}[Entropy solutions of the Burgers equation]\label{thm:Entropy solutions}
Let $u_0 \in L^\infty(\R)$. Then there exists a unique global entropy solution of the Burgers equation \eqref{B}.
If $u_0 \in L^p(\R)$ for some $p \in [1,\infty]$, then 
\begin{equation}\label{Lp estimates}
\|u\|_{L^\infty([0,\infty); L^p(\R))}	\leqslant \|u_0\|_{L^p(\R)}
\end{equation}
If, in addition, $u_0\in H^3(\R)\cap L^1(\R)$ is nontrivial, then $u(t,\cdot)$ has a jump for any $t>T_{\mathrm B}^*$. 
Moreover, for any $t>  T_{\mathrm B}^*$, we have the loss of the energy 
\begin{equation*}
\|u(t)\|_{L^2(\R)}^2 < \|u_0\|_{L^2(\R)}^2 .
\end{equation*}
\end{thm}
%%%%%%%%%%%%%%%%%%%%%%%%%%%%%%%%%%%%%%%%%%%%%%%%%%%%%%%%%%%%%%%%%%%%%%%%

%%%%%%%%%%%%%%%%%%%%%%%%%%%%%%%%%%%%%%%%%%%%%%%%%%%%%%%%%%%%%%%%%%%%%%%%
\begin{proof}
Existence, uniqueness, and \eqref{Lp estimates} follow from the classical entropy theory, see \cite{Kruzhkov1970,Bressan2000,Dafermos2016}.
We give the argument for the shock formation and the strict energy loss.

%%%%%%%%%%%%%%%%%%%%%%%%%%%%%%%%%%%%%%%%%%%%%%%%%%%%%%%%%%%%%%%%%%%%%%%%
\medskip
\noindent\textbf{Step 1. Shock formation.}
By the Lax--Oleinik formula, if $y(t,x)$ minimizes
\begin{equation}\label{Min}
	y \mapsto \int_0^y u_0(z)\, \ud z + \tfrac{(x-y)^2}{6t}, 
\end{equation}
then, for almost every $x$, the unique entropy solution is given by 
\begin{equation*}
	u(t,x) = \tfrac{x-y(t,x)}{3 t}.
\end{equation*}	
Since $y(t,x)$ is a minimizer of \eqref{Min}, it satisfies
\begin{equation}\label{xy relation}
	u_0(y(t,x)) = \tfrac{x-y(t,x)}{3t} \quad \iff \quad x= 3t u_0(y(t,x)) + y(t,x).
\end{equation}
Fix now $t>T_{\mathrm B}^*$ and assume, by contradiction, that the minimizer in \eqref{Min} is unique for every $x \in \R$.  
Then $x \mapsto y(t,x)$ is continuous. 
Indeed, for a sequence $x_n \to x$, the corresponding minimizers remain bounded by \eqref{xy relation}.
Every accumulation point minimizes \eqref{Min} at $x$, and uniqueness identifies it with $y(t,x)$.
Since $u_0$ is bounded, then \eqref{xy relation} implies that $y(t,x) \to \pm \infty$ as $x \to \pm \infty$ for any $t>0$. 
Therefore $y(t,\cdot)$ is surjective.

It is a standard property of the Lax--Oleinik formula that the map $x \mapsto y(t,x) $ is nondecreasing, then \eqref{xy relation} implies that it is injective. 
Then its inverse $y \mapsto 3 t u_0(y)+y$ is nondecreasing as well. 
On the other hand, since $t>T_{\mathrm B}^*$, there exists $y_0$ such that $3 t u_0'(y_0)+1<0$, which is a contradiction.

Therefore, for every $t > T_{\mathrm B}^*$, there exists $x \in \R$ for which the minimizer in \eqref{Min} is not unique. 
Let $y_-<y_+$ be the smallest and largest minimizers, the Lax--Oleinik formula gives
\begin{equation*}
	u(t,x-) = \frac{x-y_-}{3t} 	> \frac{x-y_+}{3t}	= u(t,x+).	
\end{equation*}
Thus $u(t,\cdot)$ has a discontinuous shock for every $t>T_{\mathrm B}^*$.

%%%%%%%%%%%%%%%%%%%%%%%%%%%%%%%%%%%%%%%%%%%%%%%%%%%%%%%%%%%%%%%%%%%%%%%%
\medskip
\noindent\textbf{Step 2. Strict energy loss.}
We prove that
\begin{equation*}
    \|u(t)\|_{L^2(\R)}^2 < \|u_0\|_{L^2(\R)}^2
    \qquad \text{for every } t>T_{\mathrm B}^*.
\end{equation*}
Recall that the flux and the quadratic entropy pair are
\begin{equation*}
    f(u) = \frac32 u^2, \qquad \eta(u) = \frac12 u^2, \qquad q(u)=u^3.
\end{equation*}
Entropy solutions of the Burger equation satisfy \cite[Theorem 6.2.6]{Dafermos2016}  
\begin{equation*}
    u \in BV_{\mathrm{loc}} ((0,\infty) \times \R), \qquad \text{and} \qquad u(t,\cdot) \in BV_{\mathrm{loc}} (\R)\ \text{for any } t>0.
\end{equation*}
The chain rule for BV functions therefore applies to the entropy pair $(\eta,q)$ \cite{Dafermos2016}.
It shows that the entropy production is concentrated on the jump set of $u$. 
Indeed, the continuous part vanishes by the equation $u_t + f(u)_x = 0$, see \cite[Theorem 11.13.1]{Dafermos2016}.

At a shock, denote the left and right traces by $u^-$ and $u^+$.
Since the solution is entropic, $u^->u^+$, and the Rankine--Hugoniot condition gives the shock speed
\begin{equation*}
    \sigma = \tfrac{f(u^+)-f(u^-)}{u^+-u^-} = \tfrac32 (u^-+u^+).
\end{equation*}
Writing $[g]=g^+-g^-$, we obtain
\begin{align*}
    [q] - \sigma [\eta]  &= (u^+)^3-(u^-)^3  -\tfrac34 (u^-+u^+) \bigl((u^+)^2-(u^-)^2\bigr) \\
    &= -\tfrac14 (u^--u^+)^3<0.
\end{align*}
By the entropy-production formula for uniformly convex scalar
conservation laws, recalled in \cite[Section 1]{BianchiniMarconi2017},
we have, for $0<a<t$,
\begin{equation*}
    \half \|u(t)\|_{L^2(\R)}^2 - \half \|u(a)\|_{L^2(\R)}^2 =    -\tfrac14 \int_a^t \sum_{x\in J(s)} \bigl(u(s,x-) - u(s,x+)\bigr)^3 \,\ud s,
\end{equation*}
where $J(s)$ denotes the set of jump points of $u(s,\cdot)$.
This identity follows by using Volpert's BV chain rule formula,  then integrating over $\R$ using cutoff functions and the estimation \eqref{Lp estimates} for $p=3$.
We also use that $u \in C([0,\infty); L^2(\R))$, which follows from the standard $L^1$ continuity and the $L^\infty$ bound.

Fix $t>T_{\mathrm B}^*$ and choose $a\in(T_{\mathrm B}^*,t)$.
By Step 1, a nonzero downward jump exists at every time $s\in(a,t)$, so the integral above is strictly positive.
Consequently,
\begin{equation*}
    \|u(t)\|_{L^2(\R)}^2 < \|u(a)\|_{L^2(\R)}^2 \leqslant \|u_0\|_{L^2(\R)}^2,
\end{equation*}
as claimed.
\end{proof}
%%%%%%%%%%%%%%%%%%%%%%%%%%%%%%%%%%%%%%%%%%%%%%%%%%%%%%%%%%%%%%%%%%%%%%%%

%%%%%%%%%%%%%%%%%%%%%%%%%%%%%%%%%%%%%%%%%%%%%%%%%%%%%%%%%%%%%%%%%%%%%%%%
\subsection{The Camassa--Holm equation}
%%%%%%%%%%%%%%%%%%%%%%%%%%%%%%%%%%%%%%%%%%%%%%%%%%%%%%%%%%%%%%%%%%%%%%%%

The local and global theory of the Camassa--Holm equation is classical, see \cite{ConstantinEscher1998Global,ConstantinEscher1998WaveBreaking,LiOlver2000,RodriguezBlanco2001,Yin2004,Yin2007}.

%%%%%%%%%%%%%%%%%%%%%%%%%%%%%%%%%%%%%%%%%%%%%%%%%%%%%%%%%%%%%%%%%%%%%%%%
\begin{thm}[Local well-posedness]\label{thm:CH-LWP}
Let $s>\frac32$, $\ell>0$ and $u_0\in H^s(\R)$. Then there exists a maximal time $T_\ell^*=T_\ell^*(u_0)>0$ such that the Camassa--Holm equation \eqref{CH} has a unique solution $u^\ell$ satisfying
\begin{equation*}
u^\ell \in C\bigl([0,T_\ell^*); H^s(\R)\bigr) \cap C^1\bigl([0,T_\ell^*); H^{s-1}(\R)\bigr).
\end{equation*}
Moreover, for any $T<T_\ell^*$, the map
\begin{equation*}
u_0 \longmapsto u^\ell
\end{equation*}
is continuous from a neighborhood of $u_0$ in $H^s(\R)$ into $C([0,T];H^s(\R))$. In addition, if $T_\ell^*<\infty$, then 
\begin{equation*}
\liminf_{t\uparrow T_\ell^*} \inf_{x\in\R}u_x^\ell(t,x) = -\infty.
\end{equation*}
\end{thm}
%%%%%%%%%%%%%%%%%%%%%%%%%%%%%%%%%%%%%%%%%%%%%%%%%%%%%%%%%%%%%%%%%%%%%%%%

The maximal existence time $T_\ell^*$ is independent of $s>3/2$, due to the persistence of the Sobolev regularity.
The local well-posedness for $s>\frac32$ is classical, see \cite{ConstantinEscher1998Global,Yin2007}, Theorems 4.4, 4.5 and 6.2 in \cite{LiOlver2000} and Theorems 3.1 and 3.2 in \cite{RodriguezBlanco2001}.

Since our goal is to study the singular limit $\ell \to 0$, a first step is to obtain a time of existence uniform in $\ell$.  For that purpose, we recall the following theorem from \cite{LiYuZhu2023}.

%%%%%%%%%%%%%%%%%%%%%%%%%%%%%%%%%%%%%%%%%%%%%%%%%%%%%%%%%%%%%%%%%%%%%%%%
\begin{thm}[Convergence to Burgers]\label{thm:Convergence}
Let $s>\frac32$ and $u_0\in H^s(\R)$.  
There exists $T=T(\|u_0\|_{H^s})>0$ such that, for every $\ell\in(0,1)$, the solutions $u^\ell$ and $u$ of the Camassa--Holm equation \eqref{CH} and the Burgers equation \eqref{B}, respectively, satisfy $u^\ell,u\in C\bigl([0,T];H^s(\R)\bigr)$.
Moreover, it holds that 
\begin{equation*}
	\lim_{\ell \to 0} \|u^\ell-u\|_{L^\infty([0,T];H^s(\R))}=0.
\end{equation*}
\end{thm}
%%%%%%%%%%%%%%%%%%%%%%%%%%%%%%%%%%%%%%%%%%%%%%%%%%%%%%%%%%%%%%%%%%%%%%%%

This theorem gives a time $T>0$, independent of $\ell$, on which the CH solutions converge strongly to the classical Burgers solutions as $\ell \to 0$.
The aim of this paper is to show that such strong convergence cannot in general be continued beyond the Burgers blow-up time.
To that end, we recall the following sufficient condition for wave breaking (see Theorem 4.2 in \cite{ConstantinEscher1998WaveBreaking}).

%%%%%%%%%%%%%%%%%%%%%%%%%%%%%%%%%%%%%%%%%%%%%%%%%%%%%%%%%%%%%%%%%%%%%%%%
\begin{thm}[A sufficient condition for wave breaking]\label{thm:Blow-up}
Assume $u_0 \in H^3(\R)$. If there exists $x_0 \in \R$ such that
\begin{equation*}
u_0'(x_0) < -\frac{1}{\sqrt{2} \ell^{3/2}} \left(\int_\R (u_0^2 + \ell^2 (u_0')^2)\, \ud x \right)^{1/2},
\end{equation*}
then the corresponding solution of \eqref{CH} blows up in finite time.
\end{thm}
%%%%%%%%%%%%%%%%%%%%%%%%%%%%%%%%%%%%%%%%%%%%%%%%%%%%%%%%%%%%%%%%%%%%%%%%

Theorem 4.2 in \cite{ConstantinEscher1998WaveBreaking} gives a sufficient condition for wave breaking in the case $\ell=1$, see also \cite{Yin2004}. 
The general result given by Theorem \ref{thm:Blow-up} follows by a simple change of variables.  
We finally recall the global smooth existence result for the Camassa--Holm equation when the initial momentum
$m^\ell_0=u_0-\ell^2u_0''$ does not change sign, see
\cite{ConstantinEscher1998Global} and Theorem 4.1 in
\cite{RodriguezBlanco2001}.

%%%%%%%%%%%%%%%%%%%%%%%%%%%%%%%%%%%%%%%%%%%%%%%%%%%%%%%%%%%%%%%%%%%%%%%%
\begin{thm}[Global existence for one-sign momentum] \label{thm:CH-global-momentum}
Let $\ell>0$, $s>\frac32$, and $u_0\in H^s(\R)\cap L^1(\R)$. 
Assume that the initial momentum $m_0^\ell$ does not change sign (in the sense of distributions).
Then, the solution given by Theorem \ref{thm:CH-LWP} exists globally, $m^\ell$ has the same sign for all $t>0$, and
\begin{equation*}
	u^\ell \in C([0,\infty); H^s(\R)) \cap C^1([0,\infty); H^{s-1}(\R)).
\end{equation*}
\end{thm}
%%%%%%%%%%%%%%%%%%%%%%%%%%%%%%%%%%%%%%%%%%%%%%%%%%%%%%%%%%%%%%%%%%%%%%%%

%%%%%%%%%%%%%%%%%%%%%%%%%%%%%%%%%%%%%%%%%%%%%%%%%%%%%%%%%%%%%%%%%%%%%%%%
\section{Proof of the main result}\label{sec:proof}
%%%%%%%%%%%%%%%%%%%%%%%%%%%%%%%%%%%%%%%%%%%%%%%%%%%%%%%%%%%%%%%%%%%%%%%%

%%%%%%%%%%%%%%%%%%%%%%%%%%%%%%%%%%%%%%%%%%%%%%%%%%%%%%%%%%%%%%%%%%%%%%%%
\subsection*{Uniform estimates}
%%%%%%%%%%%%%%%%%%%%%%%%%%%%%%%%%%%%%%%%%%%%%%%%%%%%%%%%%%%%%%%%%%%%%%%%
Throughout this section, $u_0$ satisfies the assumptions of Theorem \ref{thm:main} and $\ell \in (0,\ell_0)$. 
We use the conserved quantities of the Camassa--Holm equation to derive estimates independent of $\ell$.

We start by writing \eqref{momentum eq CH} in the form
\begin{equation*}
m_t^\ell + (u^\ell m^\ell)_x + \half [(u^\ell)^2]_x -\half \ell^2 [(u_x^\ell)^2]_x = 0.
\end{equation*}
We recall that $m^\ell$ has the same sign for all time $t \geqslant 0$. 
Integrating over $\R$,  we obtain 
\begin{equation}\label{L1 pm}
	\int_\R m^\ell(t) \, \ud x = \int_\R  u^\ell(t)\, \ud x = 	\int_\R  u_0\, \ud x.
\end{equation}
Recall now that 
\begin{equation*}
u^\ell = G_\ell \ast m^\ell, \qquad G_\ell(x) = \frac{1}{2\ell} \mathrm{e}^{-|x|/\ell}.
\end{equation*}
Since $m^\ell$ has one sign, so does $u^\ell$. 
Therefore, \eqref{L1 pm} becomes
\begin{equation}\label{L1}
	\|m^\ell(t)\|_{L^1(\R)} = \|u^\ell(t)\|_{L^1(\R)} = 	\|u_0\|_{L^1(\R)}.
\end{equation}
In addition, we have the pointwise bound
\begin{equation}\label{u_x u}
	\ell |u^\ell_x| = \ell |G_\ell' \ast m^\ell| \leqslant \ell |G_\ell'| \ast |m^\ell| =  |G_\ell \ast m^\ell| = |u^\ell|.
\end{equation}
Differentiating \eqref{CH} with respect to $x$, we obtain 
\begin{equation}\label{CHx}
u_{tx}^\ell + u^\ell u_{xx}^\ell + \half (u_x^\ell)^2 +  P^\ell + \tfrac{Q^\ell- (u^\ell)^2}{\ell^2}  = 0.
\end{equation}
Multiplying now \eqref{CH} by $2 u^\ell$ and \eqref{CHx} by $2 \ell^2 u_x^\ell$, we obtain the local energy conservation law 
\begin{equation}\label{energy local equation}
	\left((u^\ell)^2+\ell^2 (u^\ell_x)^2 \right)_t + \left( 2 \ell^2 u^\ell P^\ell + 2 u^\ell Q^\ell + \ell^2 u^\ell (u^\ell_x)^2 \right)_x = 0.
\end{equation}
Integrating over $\R$, one obtains for any $t>0$ that
\begin{equation}\label{energy equality}
	\int_\R ((u^\ell)^2+\ell^2 (u_x^\ell)^2)\, \ud x = \int_\R (u_0^2+\ell^2 (u'_0)^2)\, \ud x \leqslant \|u_0\|^2_{H^1(\R)} .
\end{equation}
The cubic Hamiltonian
\begin{equation*}
	 H_\ell(t) \eqdef \int_{\mathbb R} \left((u^\ell)^3 + \ell^2 u^\ell(u^\ell_x)^2\right) \ud x
\end{equation*}
is conserved \cite{CamassaHolm1993}. 
Since $u^\ell$ has one sign, we obtain
\begin{align}\nonumber
\| u^\ell(t) \|_{L^3(\R)}^3 &=\int_{\mathbb R}|u^\ell(t)|^3\, \ud x \leqslant \left| H_\ell(t)\right| = \left| H_\ell(0)\right| \\ \label{L3}
&= \int_{\mathbb R} \left( |u_0|^3 + \ell^2 |u_0| (u'_0)^2 \right) \ud x \leqslant 2 \int_{\mathbb R} |u_0|^3 \ud x,
\end{align}
where the last inequality follows from $\ell^2 (u'_0)^2 \leqslant |u_0|^2$.
Using \eqref{L1}, \eqref{u_x u} and \eqref{L3} we deduce that for any $p \in [1,3]$, it holds that
\begin{equation}\label{uniform estimates}
	\|u^\ell\|_{L^\infty([0,\infty); L^p(\R))} + \|\ell u_x^\ell\|_{L^\infty([0,\infty); L^p(\R))} \lesssim_{p,u_0} 1.
	\end{equation} 
We now show that the energy flux in \eqref{energy local equation} is uniformly bounded in $L^\infty([0,\infty); L^1(\R))$. Indeed, since $\|G_\ell\|_{L^1(\R)} = 1$, we have 
\begin{align*}
	\| \ell^2 u^\ell P^\ell  \|_{L^1(\R)} &\leqslant \|u^\ell\|_{L^3(\R)} \| \ell^2  P^\ell  \|_{L^{3/2}(\R)} \leqslant \|u^\ell\|_{L^3(\R)} \| \ell  u_x^\ell  \|^2_{L^3(\R)} \lesssim_{u_0} 1, \\
	\| u^\ell Q^\ell  \|_{L^1(\R)} &\leqslant \|u^\ell\|_{L^3(\R)} \| Q^\ell  \|_{L^{3/2}(\R)} \leqslant \|u^\ell\|_{L^3(\R)}^3 \lesssim_{u_0} 1, \\
	\| \ell^2 u^\ell (u^\ell_x)^2   \|_{L^1(\R)} &\leqslant \|u^\ell\|_{L^3(\R)}^3 \lesssim_{u_0} 1.
\end{align*}
Summing up, we obtain 
\begin{equation}\label{energy flux}
\left\| 2 \ell^2 u^\ell P^\ell + 2 u^\ell Q^\ell + \ell^2 u^\ell (u^\ell_x)^2  \right\|_{L^\infty([0,\infty); L^1(\R))}	\lesssim_{u_0} 1.
\end{equation}

%%%%%%%%%%%%%%%%%%%%%%%%%%%%%%%%%%%%%%%%%%%%%%%%%%%%%%%%%%%%%%%%%%%%%%%%
\subsection*{Energy tightness and strong $L^2$ convergence}
%%%%%%%%%%%%%%%%%%%%%%%%%%%%%%%%%%%%%%%%%%%%%%%%%%%%%%%%%%%%%%%%%%%%%%%%

Choose $\chi_R \in C^\infty(\R)$ such that
\begin{equation*}
	0 \leqslant \chi_R \leqslant 1,  \qquad \chi_R = 0\ \text{on } |x| \leqslant R/2,  \qquad \chi_R = 1\ \text{on } |x| \geqslant R, \qquad \text{and} \qquad |\chi_R'| \leqslant C/R.
\end{equation*}
Multiplying \eqref{energy local equation} by $\chi_R $ and using \eqref{energy flux}, one gets 
\begin{equation*}
	\frac{\ud\ }{\ud t} \int_\R \left((u^\ell)^2+\ell^2 (u^\ell_x)^2 \right) \chi_R\, \ud x \lesssim \frac{1}{R}.
\end{equation*}
Integrating over $(0,t)$, for $t_2 \geqslant t$, gives
\begin{equation}\label{tightness}
	\int_{\{|x| \geqslant R\}} \left((u^\ell)^2+\ell^2 (u^\ell_x)^2 \right) \ud x  \leqslant \int_{\{|x| \geqslant R/2 \}}  \left(u_0^2+\ell_0^2 (u'_0)^2 \right) \ud x  + \frac{C t_2}{R}.
\end{equation}
The right-hand side goes to zero as $R \to \infty$, uniformly in $\ell \in (0,\ell_0)$ and $t \in [0,t_2]$.

Assume, for a contradiction, that \eqref{convergence in L1} holds for some $t_2>t_1 \geqslant T_{\mathrm B}^*$ and some subsequence $\ell_n \to 0$. 
For the sake of simplicity, we write $u^\ell$ for this subsequence. 
Using \eqref{convergence in L1}, \eqref{uniform estimates} and \eqref{Lp estimates}, we deduce that the convergence holds in $L^2_{\mathrm{loc}}([t_1,t_2] \times \R)$
\begin{equation*}
u^{\ell} \to u \qquad \text{strongly in } L^2_{\mathrm{loc}}\bigl([t_1,t_2] \times \R \bigr).
\end{equation*}
Moreover, 
\begin{equation*}
	\|u^\ell - u\|_{L^2([t_1,t_2] \times \R)}^2 \leqslant  \int_{t_1}^{t_2} \int_{\{|x| < R\}} |u^\ell - u|^2 \,  \ud x\, \ud t + 
	2 \int_{t_1}^{t_2} \int_{\{|x| \geqslant R\}} \left((u^\ell)^2 + (u)^2 \right)  \ud x\, \ud t.
\end{equation*}
Taking $\ell \to 0$ and then $R \to \infty$ and using \eqref{tightness} with \eqref{Lp estimates}, we deduce that 
\begin{equation}\label{convergence in L2}
u^{\ell} \to u \qquad \text{strongly in } L^2 \bigl([t_1,t_2] \times \R \bigr).
\end{equation}

%%%%%%%%%%%%%%%%%%%%%%%%%%%%%%%%%%%%%%%%%%%%%%%%%%%%%%%%%%%%%%%%%%%%%%%%
\subsection*{The defect density}
%%%%%%%%%%%%%%%%%%%%%%%%%%%%%%%%%%%%%%%%%%%%%%%%%%%%%%%%%%%%%%%%%%%%%%%%

By Theorem \ref{thm:Entropy solutions}, we have 
\begin{equation*}
	\int_{t_1}^{t_2}  \int_\R u^2\, \ud x\, \ud t < (t_2-t_1)  \int_\R u_0^2\, \ud x.
\end{equation*}
Consequently, the energy conservation \eqref{energy equality} together with \eqref{convergence in L2} yields
\begin{equation}\label{total mass}
	\lim_{\ell \to 0} \int_{t_1}^{t_2} \int_\R \ell^2(u_x^\ell)^2\, \ud x \, \ud t = (t_2-t_1)  \int_\R u_0^2\, \ud x - \int_{t_1}^{t_2}  \int_\R u^2\, \ud x\, \ud t > 0.
\end{equation} 
Define 
\begin{equation*}
	w^\ell \eqdef \ell^2 (u_x^\ell)^2
\end{equation*}
which is uniformly bounded in $L^\infty([t_1,t_2]; L^{3/2}(\R))$ due to \eqref{uniform estimates}. 
Let $w$ be a weak limit of $w^\ell$ in $L^{3/2}([t_1,t_2] \times \R)$ (we may pass to a subsequence if necessary). 
Using \eqref{total mass} and the tightness of the energy \eqref{tightness}, we obtain that 
\begin{align}\nonumber
0 < \lim_{\ell \to 0}	\int_{t_1}^{t_2} \int_\R  w^\ell \, \ud x \, \ud t 
&= \lim_{R \to \infty} \lim_{\ell \to 0}	\int_{t_1}^{t_2} \int_{\{|x| < R\}} w^\ell \, \ud x \, \ud t + \lim_{R \to \infty} \lim_{\ell \to 0}	\int_{t_1}^{t_2} \int_{\{|x| \geqslant R\}} w^\ell \, \ud x \, \ud t\\ \label{total mass0}
&=\int_{t_1}^{t_2} \int_\R  w \, \ud x \, \ud t,
\end{align}
where the first limit is obtained due to the weak convergence of $w^\ell$ to $w$ in $L^{3/2}([t_1,t_2] \times \R)$, which can be tested against $1$ on compact sets. The second limit vanishes due to \eqref{tightness}.

%%%%%%%%%%%%%%%%%%%%%%%%%%%%%%%%%%%%%%%%%%%%%%%%%%%%%%%%%%%%%%%%%%%%%%%%
\subsection*{Failure of convergence to the Burgers equation}
%%%%%%%%%%%%%%%%%%%%%%%%%%%%%%%%%%%%%%%%%%%%%%%%%%%%%%%%%%%%%%%%%%%%%%%%

We now pass to the limit in \eqref{CH}. 
Let $\varphi \in C^\infty_c((t_1,t_2) \times \R )$ be a test function. 
Then 
\begin{gather*}
\int_{t_1}^{t_2} \int_\R \ell^2 P^\ell \varphi\, \ud x \, \ud t  = \int_{t_1}^{t_2} \int_\R \half w^\ell (G_\ell \ast \varphi) \, \ud x\, \ud t , \\ 
\int_{t_1}^{t_2} \int_\R Q^\ell \varphi\, \ud x \, \ud t  = \int_{t_1}^{t_2} \int_\R (u^\ell)^2 (G_\ell \ast \varphi) \, \ud x\, \ud t .
\end{gather*}
Using \eqref{uniform estimates}, \eqref{Lp estimates} and \eqref{convergence in L2}, we obtain
\begin{equation*}
	(u^\ell)^2 \to u^2 \quad \text{in } L^1([t_1,t_2] \times \R).
\end{equation*}
Since $G_\ell \ast \varphi \to \varphi$ uniformly, we obtain, as $\ell \to 0$, that
\begin{equation*}
\ell^2 P^\ell \rightharpoonup \half w, \qquad Q^\ell \rightharpoonup u^2,
\end{equation*}
in the sense of distributions. 
Passing to the limit in \eqref{CH} gives
\begin{equation*}
	u_t +  \left(\tfrac32 u^2\right )_x = - \half w_x.
\end{equation*}
Since, $u$ is the entropy solution of the Burgers equation, then $w_x=0$ on $(t_1,t_2) \times \R$.
Therefore $w$ is a function of time.
Due to the energy estimate, the integral of $w$ over $\R$ must be finite for almost all time, therefore $w=0$ on $(t_1,t_2) \times \R$.
This contradicts \eqref{total mass0} and ends the proof of Theorem \ref{thm:main}.
\qed

%%%%%%%%%%%%%%%%%%%%%%%%%%%%%%%%%%%%%%%%%%%%%%%%%%%%%%%%%%%%%%%%%%%%%%%%

%%%%%%%%%%%%%%%%%%%%%%%%%%%%%%%%%%%%%%%%%%%%%%%%%%%%%%%%%%%%%%%%%%%%%%%%

\end{document}